\documentclass[10pt,a4paper]{amsart}
\newtheorem{theorem}{Theorem}
\newtheorem{lemma}{Lemma}

\newtheorem{remark}{Remark}
\newtheorem{definition}{Definition}

\newtheorem{assumption}{Assumption}
\usepackage{amsmath}
\usepackage{commath}
\usepackage{amssymb}
\usepackage{array}
\usepackage{cite}
\usepackage{bbm}
 \usepackage{amsthm}
\begin{document}

	\title[On convergence of weighted functionals of random fields]{Weak convergence of weighted additive functionals of long-range dependent fields}
	
	\author{T. Alodat}
	\address{Department of Mathematics and Statistics La Trobe University, Melbourne, Victoria, 3086}
	\email{Alodat.t@students.latrobe.edu.au}
	\thanks{ }
	
	\author{A. Olenko}
	\address{Department of Mathematics and Statistics La Trobe University, Melbourne, Victoria, 3086}
	\email{a.olenko@latrobe.edu.au}
	\thanks{Supported in part under the Australian Research Council's Discovery Project funding scheme (project number DP160101366)}
	
	\date{03/10/2017}
	\dedicatory{This paper is dedicated to the 85th anniversary of professor M. Yadrenko's birth.}
	\begin{abstract}
		We provide asymptotic results for the distribution of weighted nonlinear functionals of Gaussian field with long-range dependence. We also show that integral functionals and the corresponding additive functionals have same distributions under certain assumptions. The result is applied to integrals over a multidimensional rectangle with a constant weight function.
	\end{abstract}
	\maketitle
\section{Introduction}
Professor M. Yadrenko (1932-2004) obtained fundamental results in the theory and statistical inference of random fields. He and A. M. Yaglom were founders of the modern spectral theory of spatial random processes. A good account of M. Yadrenko's research and main results can be found in his classical monograph~\cite{yadrenko1983spectral} and paper \cite{buldygin1992works}. In this article we present some new results for random fields on the plane. Such fields are one of the most important cases for applications in which M. Yadrenko was interested.

New technologies such as satellite imaging, positron emission tomography and functional magnetic resonance imaging, have provided various spatial data with strong dependence structures. Random fields are flexible mathematical tools to model such data. In this framework, researchers pay particular attention to various specific cases of random fields due to their mathematical tractability. For example, to model statistical properties of various datasets; the second order stationary random fields with long-range dependence have been used in many applications, such as finance, physics, environmental sciences, hydrology, telecommunications and signal processing, see \cite{beran1994statistics,doukhan2002theory,doukhan2002long,palma2007long} and the references therein.

Limit theorems play an important role in many areas of the theory of random fields and stochastic processes. In particular, the asymptotic behaviour of integrals or sums of non-linear functionals of Gaussian processes under long-range dependence have attracted much attention, see~\cite{anh2015rate,anh2017rate,taqqu1975weak, taqqu1979convergence,dobrushin1979non,doukhan2002theory,olenko2013limit,olenko2010limit,bai2013multivariate} and the references therein. The literature on this topic shows numerous examples in which the Hermite expansion was proposed as a suitable tool. In fact, it has been demonstrated that the long-range dependent summands can produce different normalising coefficients and non-Gaussian limits, that are called Hermite or Hermite-Rosenblatt distributions. These results were first obtained by Rosenblatt~\cite{rosenblatt1961independence}. Some classical approaches in asymptotic theory of functionals of random processes and fields with long-range dependence are listed below.

Taqqu~\cite{taqqu1979convergence,taqqu1975weak} discussed an asymptotic behaviour of functionals of the first two Hermite ranks of stationary Gaussian processes. He proved that the limiting processes are non-Gaussian and depend on the Hermite rank of the functionals. Dobrushin and Major~\cite{dobrushin1979non} showed that the normalised sums of stationary random variables are weakly convergent to some self-similar processes that are given in terms of stochastic multiple Wiener-It{\^o} integrals. Furthermore, Rosenblatt~\cite{rosenblatt1981limit} derived limit theorems for partial sums of a non-linear functional of strongly dependent stationary Gaussian sequences. Pipiras and Taqqu~\cite{pipiras2010regularization} presented a different proof for the finite time interval representation of Hermite processes. The result was obtained by regularising both Hermite processes and the fractional Brownian motion instead of cumulants and convergence of partial sums. Pakkanen and R\'{e}veillac~\cite{pakkanen2016functional} derived limit theorems for generalised variations of the fractional Brownian sheet with a general Hermite rank under long-range dependence. They showed that the limiting distribution is a fractional Brownian sheet that is independent and different from the original one. Bai and Taqqu~\cite{bai2013multivariate} applied a multilinear polynomial-form process with regularly varying coefficients to a sequence of independent and identically distributed random variables. They showed that the limit of the normalised partial sums would be a multivariate Gaussian process, or a multivariate Hermite process, or a mixture of both.

Some limit theorems were obtained for weighted functionals of random fields and their ramifications under long-range dependence, see~\cite{ivanov1989statistical,leonenko1999limit, olenko2013limit}. However, much of these asymptotic results are based on either characteristic functions or stochastic integrals representations, see~\cite{leonenko2006weak, leonenko1999limit, taqqu1979convergence, dobrushin1979non, major1981multiple}. The paper by Leonenko and Taufer~\cite{leonenko2006weak} considered an asymptotic distribution of functionals of Gaussian processes with long-range dependence. The asymptotic result was obtained via the characteristic function approach. 

In this paper we study limiting distributions of sums of random fields. Firstly, we prove that sums have the same distribution as the corresponding integrals for the continuous case. Then we demonstrate how to apply this result. The main tool is Lemma 1 in~\cite{leonenko2006weak}, which is modified for the two-dimensional and multidimensional cases. In the paper, we also simplify and clarify conditions and some parts of~\cite{leonenko2006weak}. Also, note that our result is a refined version of Lemma~1 in~\cite{leonenko2006weak} as we consider long-range dependent random fields with general covariance functions (satisfying Assumption 1) while~\cite{leonenko2006weak} only studied the case $B(t)=(1+t^2)^{-\alpha/2}, 0<\alpha<1$. The obtained results can be applied to more general settings than integrals over homothetic regions considered in~\cite{anh2017rate,anh2015rate,leonenko2014sojourn}.

The rest of the paper is organised as follows. In section 2 we present some basic definitions and facts of the spectral and correlation theory of random fields, which will be used extensively in the sequel. In section 3 we prove that the 
integral functionals and the corresponding sums have the same distributions. In section 4 we give a multidimensional version of the results and show their application. Finally, some discussion and open problems are presented in section 5.
\section{Long-range dependent random fields and limit theorems for their functionals}
In this section, definitions, assumptions, and basic results are presented to study the asymptotic behaviour of functionals of random fields with a singular spectrum.

In what follows we denote by $|\cdot|$ and $\Vert \cdot \Vert$ the Lebesque measure and the Euclidean distance in $\mathbb{R}^{n}$, respectively. The symbols $C$ and $\delta$ with subscripts will be used to denote constants that are not important for our discussion. Moreover, it is assumed that all random variables are defined on a fixed probability space $\left(\Omega,\mathfrak{F},\mathbb{P}\right)$.

We consider a measurable mean-square continuous zero-mean homogeneous isotropic real-valued random field $\xi\left(x\right),x\in\mathbb{R}^{n}$, with the covariance function
\begin{equation*}
	B\left(r\right)=\mathbb{E}\left(\xi(0)\xi(x)\right),\quad x\in\mathbb{R}^{n},\quad
	r=\Vert x\Vert.
\end{equation*}
It is well known that there exists a bounded nondecreasing function $\Phi\left(u\right), u\geqslant 0$, (see~\cite{ivanov1989statistical,yadrenko1983spectral}) such that
\begin{equation*}
	B\left(r\right)=\int_{0}^{\infty}Y_{n}\left(ru\right)d\Phi\left(u\right),
\end{equation*}
where the function $Y_{n}\left(\cdot\right), n\geq 1,$ is defined by
\begin{equation*}
	Y_{n}\left(u\right)=2^{(n-2)/2}\Gamma\left(\dfrac{n}{2}\right)J_{(n-2)/2}(u)u^{(2-n)/2},\quad u\geqslant 0,
\end{equation*}
where $J_{(n-2)/2}(\cdot)$ is the Bessel function of the first kind of order $(n-2)/2$, see~\cite{leonenko1999limit,yadrenko1983spectral}. The function $\Phi\left(\cdot\right)$ is called the isotropic spectral measure of the random field $\xi\left(x\right),x\in\mathbb{R}^{n}$. If there exists a function $\varphi(u), u\in[0,\infty)$, such that
\begin{equation*}
	u^{n-1}\varphi(u)\in L_{1}([0,\infty)),\quad\Phi(u)=2\pi^{n/2}/\Gamma(n/2)\int_{0}^{u}z^{n-1}\varphi(z)dz,
\end{equation*}
then the function $\varphi(\cdot)$ is called the isotropic spectral density of the field $\xi\left(x\right)$.

The field $\xi\left(x\right)$ with an absolutely continuous spectrum has the following isonormal spectral representation
\begin{equation*}
	\xi\left(x\right)=\int_{\mathbb{R}^{n}}e^{i\langle\lambda,x \rangle}\sqrt{\varphi(\Vert\lambda\Vert)}W(d\lambda),
\end{equation*}
where $W(\cdot)$ is the complex Gaussian white noise random measure on $\mathbb{R}^{n}$, see~\cite{ivanov1989statistical,leonenko1999limit,yadrenko1983spectral}. 

The Hermite polynomials $H_{m}(x), m\geq 0,$ are defined by 
\begin{align*}
	H_{m}(x)=(-1)^{m}\exp{\left(\dfrac{x^{2}}{2}\right)}\dfrac{d^{m}}{dx^{m}}\exp{\left(-\dfrac{x^{2}}{2}\right)}.
\end{align*}
The Hermite polynomials $H_{m}(x), m\geq 0,$ form a complete orthogonal system in the Hilbert space $L_{2}\left(\mathbb{R},\phi(\omega)d\omega\right)=\left\{G:\int_{\mathbb{R}}G^{2}(\omega)\phi(\omega)d\omega<\infty\right\}$, where $\phi(\omega)$ is the probability density function of the standard normal distribution.

Note, that by (2.1.8)~\cite{leonenko1999limit} we get $\mathbb{E}\left(H_{m}(\xi(x))\right)=0$ and
\begin{equation}\label{eq6}
	\mathbb{E}\left(H_{m_{1}}(\xi(x))H_{m_{2}}(\xi(y))\right)=\delta_{m_{1}}^{m_{2}}m_{1}!B^{m_{1}}(\Vert x-y\Vert),\quad x,y\in\mathbb{R}^{2},
\end{equation}
where $H_{m}(\cdot)$ is the $m$-th Hermite polynomial and $\delta_{m_{1}}^{m_{2}}$ is the Kronecker delta function.

An arbitrary function $G(\omega)\in L_{2}\left(\mathbb{R},\phi(\omega)d\omega\right)$ possesses the mean-square convergent expansion
\begin{equation*}
	G(\omega)=\sum_{j=0}^{\infty}\dfrac{C_{j}H_{j}(\omega)}{j!},\quad C_{j}=\int_{\mathbb{R}}G(\omega)H_{j}(\omega)\phi(\omega)d\omega. 
\end{equation*}
By Parseval's identity 
\begin{equation}\label{eq8}
	\sum_{j=0}^{\infty}\dfrac{C_{j}^{2}}{j!}=\int_{\mathbb{R}}G^{2}(\omega)\phi(\omega)d\omega.
\end{equation}
\begin{definition}\label{def1}{\rm\cite{taqqu1975weak}} Let $G(\omega)\in L_{2}\left(\mathbb{R},\phi(\omega)d\omega\right)$ and there exist an integer $\kappa\geqslant 1$, such that $C_{j}=0$ for all $0<j\leq \kappa-1$, but $C_{\kappa}\neq 0$. Then $\kappa$ is called the Hermite rank of $G(\cdot)$ and is denoted by $H rankG.$
\end{definition}
\begin{definition}\label{def3}{\rm\cite{bingham1989regular}} A measurable function $L:(0,\infty)\rightarrow (0,\infty)$ is said to be slowly varying at infinity if for all $t>0$
	$$\lim_{r\rightarrow \infty} \dfrac{L(tr)}{L(r)} =1.$$
\end{definition}
\begin{assumption}\label{ass1}
	Let $\xi(x), x\in\mathbb{R}^{n}$, be a homogeneous isotropic Gaussian random field with $\mathbb{E}\xi(x)=0$ and the covariance function $B(x)$, such that $B(0)=1$ and
	$$B(x)=\mathbb{E}\left(\xi\left(0\right)\xi\left(x\right)\right)=\Vert x\Vert^{-\alpha}L\left(\Vert x\Vert\right),\quad \alpha>0,$$\
	where $L\left(\Vert\cdot\Vert\right)$ is a function slowly varying at infinity.  
\end{assumption}
If $\alpha\in\left(0,n/\kappa\right)$, where $\kappa$ is the Hermite rank given in Definition \ref{def1}, then the covariance function $B(x)$ satisfying Assumption \ref{ass1} is not integrable, which corresponds to the long-range dependence case~\cite{anh2015rate}.

The notation $\Delta\subset\mathbb{R}^{n}$ will be used to denote a Jordan-measurable convex bounded set, such that $|\Delta|>0$, and $\Delta$ contains the origin in its interior. Let $\Delta(r), r>0$ be the homothetic image of the set $\Delta$, with the centre of homothety at the origin and the coefficient $r>0$, that is $\vert\Delta(r)\vert=r^{n}\vert\Delta\vert$. Let $G(\omega)\in L_{2}\left(\mathbb{R},\phi(\omega)d\omega\right)$ and denote the random variables $K_{\kappa}$ and $K_{r,\kappa}$ by
$$K_{\kappa}=\int_{\bigtriangleup(r)}G\left(\xi\left(x\right)\right)dx\quad and\quad K_{r,\kappa}=\dfrac{C_{\kappa}}{\kappa!}\int_{\bigtriangleup(r)}H_{\kappa}\left(\xi\left(x\right)\right)dx,$$
where $C_{\kappa}$ satisfies (\ref{eq8}).
\begin{theorem}\label{theo1}{\rm\cite{leonenko2014sojourn}}
	Suppose that $\xi\left(x\right), x\in\mathbb{R}^{n}$, satisfies Assumption {\rm\ref{ass1}} and $H rankG(\cdot)=\kappa\geq 1$. If a limit distribution exists for at least one of the random variables
	$$\dfrac{K_{r}}{\sqrt{Var K_{r}}}\quad and\quad \dfrac{K_{r,\kappa}}{\sqrt{Var K_{r,\kappa}}},$$
	then the limit distribution of the other random variable also exists, and the limit distributions coincide when $r\rightarrow \infty$. 
\end{theorem}
By Theorem \ref{theo1} it is enough to study $K_{r,\kappa}$ to get asymptotic distributions of $K_{\kappa}$. Therefore, we restrict our attention only to $K_{r,\kappa}$.
\begin{assumption}\label{ass2}
	The random field $\xi\left(x\right), x\in\mathbb{R}^{n}$, has the spectral density
	\begin{equation*}
		f\left(\Vert \lambda\Vert\right)=c_{1}\left(n,\alpha\right)\Vert\lambda\Vert^{\alpha-n}L\left(1/\Vert\lambda\Vert\right),
	\end{equation*}
	where $c_{1}\left(n,\alpha\right)=\Gamma\left((n-\alpha)/{2}\right)/2^{\alpha}\pi^{n/2}\Gamma\left(\alpha/2\right),$ and $L(\Vert\cdot\Vert)$ is a locally bounded function which is slowly varying at infinity.
\end{assumption}
One can find more details on relations between Assumptions~\ref{ass1} and \ref{ass2} in \cite{anh2017rate, anh2015rate}.

The function $K_{\Delta}\left(x\right)$ will be used to denote the Fourier transform of the indicator function of the set $\Delta$, i.e,
\begin{equation}\label{eq011}
	K_{\Delta}\left(x\right)=\int_{\Delta}e^{i\langle u,x \rangle}du,\quad x\in\mathbb{R}^{n}.
\end{equation}  
\begin{lemma}\label{lem1}{\rm\cite{leonenko2014sojourn}}
	If $\tau_{1},\dots,\tau_{\kappa},\kappa\geq 1$, are positive constants such that it holds $\sum_{i=1}^{\kappa}\tau_{i}<n$, then
	$$\int_{\mathbb{R}^{n\kappa}}|K_{\Delta}\left(\lambda_{1}+\cdots+\lambda_{\kappa}\right)|^{2}\dfrac{d\lambda_{1}\cdots d\lambda_{\kappa}}{\Vert\lambda_{1}\Vert^{n-\tau_{1}}\cdots\Vert\lambda_{\kappa}\Vert^{n-\tau_{\kappa}}}<\infty.$$
\end{lemma}
\begin{theorem}\label{theo2}{\rm\cite{leonenko2014sojourn}}
	Let $\xi\left(x\right), x\in\mathbb{R}^{n}$, be a homogeneous isotropic Gaussian random field with $\mathbb{E}\xi\left(x\right)=0$. If Assumptions {\rm\ref{ass1}} and {\rm\ref{ass2}} hold, then for $r \rightarrow\infty$ the random variables 
	$$X_{r,\kappa}(\Delta)=r^{\kappa\alpha/2-n}L^{-\kappa/2}(r)\int_{\Delta(r)}H_{\kappa}\left(\xi(x)\right)dx,$$\
	converge weakly to  
	\begin{equation}\label{eq11}
		X_{\kappa}(\Delta)=c_{1}^{\kappa/2}(n,\alpha)\int_{\mathbb{R}^{n\kappa}}^{'}K_{\Delta}\left(\lambda_{1}+\cdots+\lambda_{\kappa}\right)\dfrac{W(d\lambda_{1})\cdots W(d\lambda_{\kappa})}{\Vert\lambda_{1}\Vert^{(n-\alpha)/2}\cdots\Vert\lambda_{\kappa}\Vert^{(n-\alpha)/2}},
	\end{equation}
	where $\int_{\mathbb{R}^{n\kappa}}^{'}$ denotes the multiple Wiener-It{\^o} integral.
\end{theorem}
\section{Asymptotic distribution of weighted functionals}
There are numerous papers on non-central limit theorems either for integrals or additive functionals of random fields, see, for example, \cite{anh2017rate, anh2015rate, dobrushin1979non,doukhan2002theory,  doukhan2002long, ivanov1989statistical, leonenko1999limit, leonenko2014sojourn, leonenko2006weak, major1981multiple, olenko2013limit, olenko2010limit, rosenblatt1981limit, rosenblatt1961independence, taqqu1979convergence, taqqu1975weak,yadrenko1983spectral}. The presented results below give a rigorous proof that under rather general assumptions, limits coincide for the above functionals. So, numerous existing results can be translated from continuous to discrete settings and vice versa without laborious proofs.

In this section, we present some generalisation of results in~\cite{leonenko2006weak} to random fields on the plane. The main objective of this section is to investigate the integral functional
\begin{equation*}
	X_{m}^{*}\left( T_{1},T_{2}\right)=\dfrac{1}{d_{T_{1},T_{2}}}\int_{0}^{T_{1}}\int_{0}^{T_{2}}g(t_{1},t_{2})H_{m}(\xi(t_{1},t_{2}))dt_{2}dt_{1},
\end{equation*}
as $T_{1},T_{2}\rightarrow\infty$, where  $g(t_{1},t_{2})$ is a non-random function on $[0,T_{1}]\times[0,T_{2}]$, $d_{T_{1},T_{2}}^{-1}$ is a normalizing constant and $H_{m}(\cdot)$ is the $m$-th Hermite polynomial with the leading coefficient equals to one and $\xi(t_{1},t_{2})$ is a homogeneous isotropic random field on $\mathbb{R}^{2}$. We first show that, as $T_{1},T_{2}\rightarrow\infty$, $X_{m}^{*}\left( T_{1},T_{2}\right)$ has the same distribution as the corresponding sums
\begin{equation*}
	\tilde{X}_{m}^{*}\left( T_{1},T_{2}\right)=\dfrac{1}{d_{T_{1},T_{2}}}\sum_{i=0}^{T_{1}-1}\sum_{j=0}^{T_{2}-1}g(i,j)H_{m}(\xi(i,j)),
\end{equation*}
for the discrete case.
\begin{assumption}\label{ass3} Let $g(t_{1},t_{2}), t_{1},t_{2}\in\mathbb{R}$, be such that $T^{4-\alpha m}g^{2}\left(T,T \right)L^{m}(T)\rightarrow\infty,$ as $T\rightarrow\infty,$ and there exist a function $g^{*}(u,v)$ such that
	$$\lim_{T\rightarrow\infty}\bigg\vert\dfrac{g\left(Tu,Tv\right)}{g\left(T,T\right)}-g^{*}(u,v)\bigg\vert\rightarrow 0$$
	uniformly on $u,v\in [0,1]$.
\end{assumption}
\begin{remark}\label{rem1} It follows from Assumption~{\rm\ref{ass3}} that $g^{*}(u,v)$ is bounded on $[0,1]^{2}$.
\end{remark}
\begin{remark}\label{rem2} Note that the conditions on the function $g(\cdot,\cdot)$ in Assumption~{\rm\ref{ass3}} meet various types of functions that are important in solving many statistical problems, in particular, non-linear regression and M estimators. For example, the functions $g(t_{1},t_{2})=t_{1}^{\mu_{1}}t_{2}^{\mu_{2}}$ with $g^{*}(u,v)=u^{\mu_{1}}v^{\mu_{2}}$ and
	$g(t_{1},t_{2})=t_{1}t_{2}\log(\mu_{1}+t_{1})\log(\mu_{2}+t_{2})$ (for some appropriate values of the constants $\mu_{1}$ and $\mu_{2}$) with $g^{*}(u,v)=uv$ can be considered. The case of $g(t_{1},t_{2})=const$ corresponds to classical non-weighted functionals and non-central limit theorems.
\end{remark}
\begin{remark}\label{rem3} To avoid the degenerate cases, the condition $T^{4-\alpha m}g^{2}\left(T,T\right)L^{m}(T)\rightarrow\infty,$ as $T\rightarrow\infty,$ is essential to guarantee the boundedness of the variance of $X_{m}^{*}\left( T_{1},T_{2}\right)d_{T_{1},T_{2}}^{-1}$.
\end{remark}
\begin{theorem}\label{theo3} Let $\tilde{T}=\max({T_{1},T_{2}})$. If Assumptions~{\rm\ref{ass1}} and {\rm\ref{ass3}} hold, and there exist 
	$\lim_{\substack{\tilde{T}\rightarrow\infty}}T_{i}/\tilde{T}, i=1,2,$ then 
	\begin{align}\label{eq13}
		\lim_{\substack{\tilde{T}\rightarrow\infty}}\dfrac{\mathbb{E}\left[\int_{0}^{T_{1}}\int_{0}^{T_{2}}g(t_{1},t_{2})H_{m}(\xi(t_{1},t_{2}))dt_{2}dt_{1}-\sum_{i=0}^{[T_{1}]-1}\sum_{j=0}^{[T_{2}]-1}g(i,j)H_{m}(\xi(i,j))\right]^{2}}{\tilde{T}^{4-\alpha m}g^{2}(\tilde{T},\tilde{T})L^{m}(\tilde{T})}=0,
	\end{align}
	where $0<\alpha<2/m$.
\end{theorem}
\begin{proof}
	Note that we can estimate the numerator in (\ref{eq13}) as
	\begin{align*}
		\mathbb{E}&\bigg[\int_{[T_{1}]}^{T_{1}}\int_{0}^{T_{2}}g(t_{1},t_{2})H_{m}(\xi(t_{1},t_{2}))dt_{1}dt_{2}+\int_{0}^{[T_{1}]}\int_{[T_{2}]}^{T_{2}}g(t_{1},t_{2})H_{m}(\xi(t_{1},t_{2}))dt_{1}dt_{2}\\
		&+\int_{0}^{[T_{1}]}\int_{0}^{[T_{2}]}g(t_{1},t_{2})H_{m}(\xi(t_{1},t_{2}))dt_{1}dt_{2}-\sum_{i=0}^{[T_{1}]-1}\sum_{j=0}^{[T_{2}]-1}g(i,j)H_{m}(\xi(i,j)\bigg]^{2}\\
		&\leq 4\mathbb{E}\bigg[\int_{[T_{1}]}^{T_{1}}\int_{0}^{T_{2}}g(t_{1},t_{2})H_{m}(\xi(t_{1},t_{2}))dt_{1}dt_{2}\bigg]^{2}\\
		&+4\mathbb{E}\bigg[\int_{0}^{[T_{1}]}\int_{[T_{2}]}^{T_{2}}g(t_{1},t_{2})H_{m}(\xi(t_{1},t_{2}))dt_{1}dt_{2}\bigg]^{2}\\
		&+2\mathbb{E}\bigg[\int_{0}^{[T_{1}]}\int_{0}^{[T_{2}]}g(t_{1},t_{2})H_{m}(\xi(t_{1},t_{2}))dt_{1}dt_{2}-\sum_{i=0}^{[T_{1}]-1}\sum_{j=0}^{[T_{2}]-1}g(i,j)H_{m}(\xi(i,j)\bigg]^{2}.
	\end{align*}
	By (\ref{eq6}), we get
	\begin{align}\label{eq14}
		\dfrac{\mathbb{E}\bigg[\int_{[T_{1}]}^{T_{1}}\int_{0}^{T_{2}}g(t_{1},t_{2})H_{m}(\xi(t_{1},t_{2}))dt_{1}dt_{2}\bigg]^{2}}{\tilde{T}^{4-\alpha m}g^{2}(\tilde{T},\tilde{T})L^{m}(\tilde{T})}=\int_{[T_{1}]}^{T_{1}}\int_{[T_{1}]}^{T_{1}}\int_{0}^{T_{2}}\int_{0}^{T_{2}}g(t_{1},t_{2})g(t^{'}_{1},t^{'}_{2})\notag\\
		\times\dfrac{B^{m}(\Vert(t_{1}-t^{'}_{1},t_{2}-t^{'}_{2})\Vert)dt_{1}dt^{'}_{1}dt_{2}dt^{'}_{2}}{\tilde{T}^{4-\alpha m}g^{2}(\tilde{T},\tilde{T})L^{m}(\tilde{T})}\notag\\
		\leq\int_{[T_{1}]}^{T_{1}}\int_{[T_{1}]}^{T_{1}}\int_{0}^{1}\int_{0}^{1}\dfrac{\vert g(t_{1},u_{2}T_{2})g(t^{'}_{1},u^{'}_{2}T_{2})\vert dt_{1}dt^{'}_{1}du_{2}du^{'}_{2}}{\tilde{T}^{2-\alpha m}g^{2}(\tilde{T},\tilde{T})L^{m}(\tilde{T})}\notag\\
		\leq\dfrac{1}{\tilde{T}^{2-\alpha m}L^{m}(\tilde{T})}\sup_{\substack{u_{1},u_{2},    u_{1}',u_{2}'\in[0,1]\\}}\dfrac{\vert g(u_{1}\tilde{T},u_{2}\tilde{T})g(u^{'}_{1}\tilde{T},u^{'}_{2}\tilde{T})\vert}{g^{2}(\tilde{T},\tilde{T})}.
	\end{align}
	By Assumption \ref{ass3} for an arbitrary $\epsilon>0$ there exists $\tilde{T}_{0}$ such that for $\tilde{T}>\tilde{T}_{0}$
	\begin{align*}
		\sup_{\substack{u_{1},u_{2}\in[0,1]\\}}\vert\dfrac{g(u_{1}\tilde{T},u_{2}\tilde{T})}{g(\tilde{T},\tilde{T})}\vert\leq\sup_{\substack{u_{1},u_{2}\in[0,1]\\}}\vert
		g^{*}(u_{1},u_{2})\vert+\epsilon.
	\end{align*}
	Hence, the upper bound of (\ref{eq14}) approaches $0$ when $\tilde{T}\longrightarrow\infty$.
	
	Similarly, one obtains that
	\begin{align*}
		\dfrac{\mathbb{E}\bigg[\int_{0}^{[T_{1}]}\int_{[T_{2}]}^{T_{2}}g(t_{1},t_{2})H_{m}(\xi(t_{1},t_{2}))dt_{1}dt_{2}\bigg]^{2}}{\tilde{T}^{4-\alpha m}g^{2}(\tilde{T},\tilde{T})L^{m}(\tilde{T})}\longrightarrow 0,
	\end{align*}
	when $\tilde{T}\longrightarrow\infty$.
	
	Hence, without lost of generality, we consider the case of integers $T_{1}$ and $T_{2}$. To simplify the calculations, let us denote the numerator in (\ref{eq13}) by $D_{T_{1},T_{2}}$, then
	\begin{align*}
		D_{T_{1},T_{2}}=\mathbb{E}\bigg[\sum_{i=0}^{T_{1}-1}\sum_{j=0}^{T_{2}-1}\int_{\left[0,1\right]^{2}}\lbrace g(x+i,y+j)H_{m}(\xi(x+i,y+j))\\-g(i,j)H_{m}(\xi(i,j))\rbrace dydx\bigg]^{2}.
	\end{align*}
	Expanding the right hand side one gets
	\begin{align}\label{eq15}
		D_{T_{1},T_{2}}=\sum_{k=1}^{3}D_{T_{1},T_{2}}^{(k)},
	\end{align}
	where
	\begin{align*}    
		D_{T_{1},T_{2}}^{(1)}&=\mathbb{E}\sum_{i,i'=0}^{T_{1}-1}\sum_{j,j'=0}^{T_{2}-1}\int_{\left[0,1\right]^{4}}g(x+i,y+j)g(x'+i',y'+j')H_{m}(\xi(x+i,y+j))\\
		&\times H_{m}(\xi(x'+i',y'+j'))dydy'dxdx',\\
		D_{T_{1},T_{2}}^{(2)}&=-2\mathbb{E}\sum_{i,i'=0}^{T_{1}-1}\sum_{j,j'=0}^{T_{2}-1}\int_{\left[0,1\right]^{2}}g(x+i',y+j')g(i,j)\\&\times H_{m}(\xi(x+i',y+j'))H_{m}(\xi(i,j))dydx,
	\end{align*}    
	and
	\begin{align*}
		D_{T_{1},T_{2}}^{(3)}=\mathbb{E}\sum_{i,i'=0}^{T_{1}-1}\sum_{j,j'=0}^{T_{2}-1}g(i,j)g(i',j')H_{m}(\xi(i,j))H_{m}(\xi(i',j')).
	\end{align*}    
	Now, using (\ref{eq6}) and Assumption 1 we can rewrite the first term in (\ref{eq15}) as follows
	\begin{align*}
		D_{T_{1},T_{2}}^{(1)} &= m!\sum_{i,i'=0}^{T_{1}-1}\sum_{j,j'=0}^{T_{2}-1}\int_{\left[0,1\right]^{4}}g(x+i,y+j)g(x'+i',y'+j')\\
		&\times B^{m}(\Vert (x'+i'-(x+i),y'+j'-(y+j))\Vert)dydy'dxdx'\\
		&=m!\int_{\left[0,T_{1}\right]^{2}}\int_{\left[0,T_{2}\right]^{2}}g(x,y)g(x',y')B^{m}(\Vert (x'-x,y'-y)\Vert)dydy'dxdx'\\
		&=m!\int_{\left[0,T_{1}\right]^{2}}\int_{\left[0,T_{2}\right]^{2}}\dfrac{g(x,y)g(x',y')L^{m}(\Vert (x'-x,y'-y)\Vert)dydy'dxdx'}{\left((x'-x)^{2}+(y'-y)^{2}\right)^{\alpha m/2}}.
	\end{align*}
	Using the following transformation
	\begin{align}\label{eq16}
		\tilde{T}u_{1}=x,\tilde{T}u_{2}=y, \tilde{T}v_{1}=x', and\quad  \tilde{T}v_{2}=y',
	\end{align}
	and elementary computations, we obtain
	\begin{align*}
		D_{T_{1},T_{2}}^{(1)}&= m!\tilde{T}^{4-m\alpha}\int_{\left[0,T_{1}\tilde{T}^{-1}\right]^{2}}\int_{\left[0,T_{2}\tilde{T}^{-1}\right]^{2}}\dfrac{g(\tilde{T}u_{1},\tilde{T}u_{2})g(\tilde{T}v_{1},\tilde{T}v_{2})}{((v_{1}-u_{1})^{2}+(v_{2}-u_{2})^{2})^{\alpha m/2}}\notag\\
		&\times L^{m}(\tilde{T}\Vert (u_{1}-v_{1},u_{2}-v_{2})\Vert)du_{1}dv_{1}du_{2}dv_{2}.
	\end{align*}
	Multiplying and dividing by $g^{2}(\tilde{T},\tilde{T})$, we obtain
	\begin{align*}
		D_{T_{1},T_{2}}^{(1)}&= m!\tilde{T}^{4-m\alpha}g^{2}(\tilde{T},\tilde{T})
		\int_{\left[0,T_{1}\tilde{T}^{-1}\right]^{2}}\int_{\left[0,T_{2}\tilde{T}^{-1}\right]^{2}}\dfrac{g(\tilde{T}u_{1},\tilde{T}u_{2})g(\tilde{T}v_{1},\tilde{T}v_{2})}{g^{2}(\tilde{T},\tilde{T})((v_{1}-u_{1})^{2}+(v_{2}-u_{2})^{2})^{\alpha m/2}}\\
		&\times L^{m}(\tilde{T}\Vert (u_{1}-v_{1},u_{2}-v_{2})\Vert)du_{1}dv_{1}du_{2}dv_{2}.
	\end{align*}
	Adding and subtracting either $g^{*}(u_{1},u_{2})$ or $g^{*}(v_{1},v_{2})$ inside the integrals, then we have
	\begin{align}\label{eq18}
		D_{T_{1},T_{2}}^{(1)}= m!\tilde{T}^{4-m\alpha}g^{2}(\tilde{T},\tilde{T})(I_{1}+2I_{2}+I_{3}),
	\end{align}
	where
	\begin{align*}
		I_{1}&=\int_{\left[0,T_{1}\tilde{T}^{-1}\right]^{2}}\int_{\left[0,T_{2}\tilde{T}^{-1}\right]^{2}}\left[\dfrac{g(\tilde{T}u_{1},\tilde{T}u_{2})}{g(\tilde{T},\tilde{T})}-g^{*}(u_{1},u_{2})\right] \left[ \dfrac{g(\tilde{T}v_{1},\tilde{T}v_{2})}{g(\tilde{T},\tilde{T})}-g^{*}(v_{1},v_{2})\right]\\
		&\times\frac{L^{m}(\tilde{T}\Vert (u_{1}-v_{1},u_{2}-v_{2})\Vert)du_{1}dv_{1}du_{2}dv_{2}}{((v_{1}-u_{1})^{2}+(v_{2}-u_{2})^{2})^{\alpha m/2}},\\
		I_{2}&=\int_{\left[0,T_{1}\tilde{T}^{-1}\right]^{2}}\int_{\left[0,T_{2}\tilde{T}^{-1}\right]^{2}}\left[\dfrac{g(\tilde{T}u_{1},\tilde{T}u_{2})}{g(\tilde{T},\tilde{T})}-g^{*}(u_{1},u_{2})\right]g^{*}(v_{1},v_{2})\\
		&\times\frac{L^{m}(\tilde{T}\Vert (u_{1}-v_{1},u_{2}-v_{2})\Vert)du_{1}dv_{1}du_{2}dv_{2}}{((v_{1}-u_{1})^{2}+(v_{2}-u_{2})^{2})^{\alpha m/2}},
	\end{align*}
	and
	\begin{align*}
		I_{3}=\int_{\left[0,T_{1}\tilde{T}^{-1}\right]^{2}}\int_{\left[0,T_{2}\tilde{T}^{-1}\right]^{2}}\dfrac{g^{*}(u_{1},u_{2})g^{*}(v_{1},v_{2})L^{m}(\tilde{T}\Vert (u_{1}-v_{1},u_{2}-v_{2})\Vert) du_{1}dv_{1}du_{2}dv_{2}}{((v_{1}-u_{1})^{2}+(v_{2}-u_{2})^{2})^{\alpha m/2}}.
	\end{align*}
	We will analyze each term $I_{i},i=1,2,3$, separately. The term $I_{1}$ can be estimated as
	\begin{align*}
		I_{1}&\leq\int_{\left[0,T_{1}\tilde{T}^{-1}\right]^{2}}\int_{\left[0,T_{2}\tilde{T}^{-1}\right]^{2}}\bigg\vert
		\dfrac{g(\tilde{T}u_{1},\tilde{T}u_{2})}{g(\tilde{T},\tilde{T})}-g^{*}(u_{1},u_{2})\bigg\vert\bigg\vert \dfrac{g(\tilde{T}v_{1},\tilde{T}v_{2})}{g(\tilde{T},\tilde{T})}-g^{*}(v_{1},v_{2})\bigg\vert\\
		&\times
		\frac{L^{m}(\tilde{T}\Vert (u_{1}-v_{1},u_{2}-v_{2})\Vert)du_{1}dv_{1}du_{2}dv_{2}}{((v_{1}-u_{1})^{2}+(v_{2}-u_{2})^{2})^{\alpha m/2}}.
	\end{align*}
	Note that $[0,T_{1}\tilde{T}^{-1}]\subseteq [0,1]$ and $[0,T_{1}\tilde{T}^{-1}]\subseteq [0,1]$. Hence,
	\begin{align*}
		I_{1}&\leq\int_{\left[0,1\right]^4}\bigg\vert
		\dfrac{g(\tilde{T}u_{1},\tilde{T}u_{2})}{g(\tilde{T},\tilde{T})}-g^{*}(u_{1},u_{2})\bigg\vert\bigg\vert\dfrac{g(\tilde{T}v_{1},\tilde{T}v_{2})}{g(\tilde{T},\tilde{T})}-g^{*}(v_{1},v_{2})\bigg\vert\\
		&\times\frac{L^{m}(\tilde{T}\Vert (u_{1}-v_{1},u_{2}-v_{2})\Vert)du_{1}dv_{1}du_{2}dv_{2}}{((v_{1}-u_{1})^{2}+(v_{2}-u_{2})^{2})^{\alpha m/2}}.
	\end{align*}
	Note that for $\alpha_{0}<1:$
	\begin{align}\label{eq19}
		\int_{\left[0,1\right]^4}\dfrac{du_{1}dv_{1}du_{2}dv_{2}}{((v_{1}-u_{1})^{2}+(v_{2}-u_{2})^{2})^{\alpha_{0}}}&\leq \int_{[0,1]^4}\dfrac{du_{1}dv_{1}du_{2}dv_{2}}{\big(2\vert v_{1}-u_{1}\vert\vert v_{2}-u_{2}\vert\big)^{\alpha_{0}}}\notag\\
		&=\dfrac{1}{2^{\alpha_{0}}}\left(\int_{[0,1]^2}\dfrac{du_{1}}{u^{\alpha_{0}}_{1}}\dfrac{dv_{1}}{\bigg\vert1-\dfrac{v_{1}}{u_{1}}\bigg\vert^{\alpha_{0}}}\right)^{2}<\infty.
	\end{align}
	Hence, we have 
	\begin{align*}
		I_{1}\leq&\sup_{\substack{u_{1}, u_{2}\in[0,1]\\}}
		\bigg\vert\dfrac{g(\tilde{T}u_{1},\tilde{T}u_{2})}{g(\tilde{T},\tilde{T})}-g^{*}(u_{1},u_{2})\bigg\vert^{2}\\
		&\times\int_{\left[0,1\right]^{4}}\frac{L^{m}(\tilde{T}\Vert (u_{1}-v_{1},u_{2}-v_{2})\Vert)du_{1}dv_{1}du_{2}dv_{2}}{((v_{1}-u_{1})^{2}+(v_{2}-u_{2})^{2})^{\alpha m/2}}.
	\end{align*}
	It follows from Assumption \ref{ass1} that $L(\cdot)$ is locally bounded and by Theorem 1.5.3~\cite{bingham1989regular} for an arbitrary $\delta>0$ there exist $\tilde{T}_{0}$ and $C>0$ such that for all $\tilde{T}>\tilde{T}_{0}$
	\begin{align*}
		\sup_{\substack{0<s<\tilde{T}\\}}\dfrac{s^{\delta}L(s)}{\tilde{T}^{\delta} L(\tilde{T})}\leq C.
	\end{align*}
	Therefore, for all $\tilde{T}>\tilde{T}_{0}$
	\begin{align*}
		\int_{\left[0,1\right]^{4}}\frac{L^{m}(\tilde{T}\Vert (u_{1}-v_{1},u_{2}-v_{2})\Vert)du_{1}dv_{1}du_{2}dv_{2}}{((v_{1}-u_{1})^{2}+(v_{2}-u_{2})^{2})^{\alpha m/2}}=\int_{\left[0,1\right]^{4}}\dfrac{(\tilde{T}\Vert (u_{1}-v_{1},u_{2}-v_{2})\Vert)^{\delta}}{\tilde{T}^{\delta}}\\
		\times\frac{L^{m}(\tilde{T}\Vert (u_{1}-v_{1},u_{2}-v_{2})\Vert)du_{1}dv_{1}du_{2}dv_{2}}{((v_{1}-u_{1})^{2}+(v_{2}-u_{2})^{2})^{\frac{\delta+m\alpha}{2}}}\\
		\leq CL^{m}(\tilde{T}) \int_{\left[0,1\right]^{4}}\dfrac{du_{1}dv_{1}du_{2}dv_{2}}{((v_{1}-u_{1})^{2}+(v_{2}-u_{2})^{2})^{\frac{\delta+m\alpha}{2}}}.
	\end{align*}
	Using this upper bound, (\ref{eq19}), and selecting $\delta$ such that $\frac{\delta+m\alpha}{2}<1$ we obtain
	\begin{align*}
		I_{1}\leq L^{m}(\tilde{T})o(1),\quad\tilde{T}\longrightarrow\infty.
	\end{align*}
	Similarly, using Remark \ref{rem1} we get
	\begin{align*}
		I_{2}\leq &\int_{\left[0,T_{1}\tilde{T}^{-1}\right]^2}\int_{\left[0,T_{2}\tilde{T}^{-1}\right]^2}\bigg\vert\dfrac{g(\tilde{T}u_{1},\tilde{T}u_{2})}{g(\tilde{T},\tilde{T})}-g^{*}(u_{1},u_{2})\bigg\vert \vert g^{*}(v_{1},v_{2})\vert\\
		&\times \dfrac{L^{m}(\tilde{T}\Vert (u_{1}-v_{1},u_{2}-v_{2})\Vert)du_{1}dv_{1}du_{2}dv_{2}}{((v_{1}-u_{1})^{2}+(v_{2}-u_{2})^{2})^{\alpha m/2}}\\
		&\leq\int_{\left[0,1\right]^4}\bigg\vert\dfrac{g(\tilde{T}u_{1},\tilde{T}u_{2})}{g(\tilde{T},\tilde{T})}-g^{*}(u_{1},u_{2})\bigg\vert \vert g^{*}(v_{1},v_{2})\vert\\
		&\times \dfrac{L^{m}(\tilde{T}\Vert (u_{1}-v_{1},u_{2}-v_{2})\Vert)du_{1}dv_{1}du_{2}dv_{2}}{((v_{1}-u_{1})^{2}+(v_{2}-u_{2})^{2})^{\alpha m/2}}\\
		& \leq \sup_{\substack{u_{1},u_{2}\in[0,1]}}
		\bigg\vert\dfrac{g(\tilde{T}u_{1},\tilde{T}u_{2})}{g(\tilde{T},\tilde{T})}-g^{*}(u_{1},u_{2})\bigg\vert\\ 
		&\times \int_{\left[0,1\right]^{4}}\frac{\vert g^{*}(v_{1},v_{2})\vert L^{m}(\tilde{T}\Vert (u_{1}-v_{1},u_{2}-v_{2})\Vert)du_{1}dv_{1}du_{2}dv_{2}}{((v_{1}-u_{1})^{2}+(v_{2}-u_{2})^{2})^{\alpha m/2}}.
	\end{align*}
	As $g^{*}(\cdot,\cdot)$ is bounded on $\left[0,1\right]^{2}$, by the same reasons as for $I_{1}$ we obtain $I_{2}\leq L^{m}(\tilde{T})o(1),$ when $\tilde{T}\longrightarrow\infty$.
	
	Note that  \[\int_{\left[0,a_{1}\right]^{2}}\int_{\left[0,a_{2}\right]^{2}}\dfrac{g^{*}(u_{1},u_{2})g^{*}(v_{1},v_{2})du_{1}dv_{1}du_{2}dv_{2}}{((v_{1}-u_{1})^{2}+(v_{2}-u_{2})^{2})^{\frac{\delta+\alpha m}{2}}}<\infty.\]
	Hence, analogously to Proposition~4.1.2~\cite{bingham1989regular} we obtain 
	\begin{align*}
		I_{3}=&\int_{\left[0,T_{1}\tilde{T}^{-1}\right]^2}\int_{\left[0,T_{2}\tilde{T}^{-1}\right]^2}\dfrac{g^{*}(u_{1},u_{2})g^{*}(v_{1},v_{2})L^{m}(\tilde{T}\Vert (u_{1}-v_{1},u_{2}-v_{2})\Vert) du_{1}dv_{1}du_{2}dv_{2}}{((v_{1}-u_{1})^{2}+(v_{2}-u_{2})^{2})^{\alpha m/2}}\\
		&\sim l_{1,2}L^{m}(\tilde{T}),\quad\tilde{T}\longrightarrow\infty,
	\end{align*}
	where
	\begin{align*}
		l_{1,2}=\int_{\left[0,a_{1}\right]^2}\int_{\left[0,a_{2}\right]^2}\dfrac{ g^{*}(u_{1},u_{2})g^{*}(v_{1},v_{2}) du_{1}dv_{1}du_{2}dv_{2}}{((v_{1}-u_{1})^{2}+(v_{2}-u_{2})^{2})^{\alpha m/2}},
	\end{align*}
	and $a_{i}=\lim_{\tilde{T}\longrightarrow\infty} T_{i}/\tilde{T}, i=1,2$.
	
	Notice, that $\vert l_{1,2}\vert<\infty$ by Remark \ref{rem1} and (\ref{eq19}). Therefore, by combining these results and (\ref{eq18}), we have
	\begin{align*}
		D_{T_{1},T_{2}}^{(1)}= m!\tilde{T}^{4-m\alpha}g^{2}(\tilde{T},\tilde{T})L^{m}(\tilde{T})(l_{1,2}+o(1)),\quad T_{1}, T_{2}\longrightarrow\infty.
	\end{align*}
	Now, we consider the second term $D_{T_{1},T_{2}}^{(2)}$:
	\begin{align*}
		D_{T_{1},T_{2}}^{(2)}&=-2m!\sum_{i,i'=0}^{T_{1}-1}\sum_{j,j'=0}^{T_{2}-1}\int\limits_{\left[0,1\right]^{2}}g(x+i',y+j')g(i,j)\\&\times B^{m}(\Vert(i-(x+i'),j-(y+j'))\Vert)dydx\\
		& =-2 m!\int_{[0,T_{1}]}\int_{[0,T_{2}]}g(x,y)\sum_{i=0}^{T_{1}-1}\sum_{j=0}^{T_{2}-1} g(i,j)B^{m}(\Vert (i-x,j-y)\Vert)dydx\\
		&=-2m!\int_{[0,T_{1}]}\int_{[0,T_{2}]}\sum_{i=0}^{T_{1}-1}\sum_{j=0}^{T_{2}-1}\dfrac{g(x,y)g(i,j)L^{m}(\Vert (i-x,j-y)\Vert)dydx}{\left((i-x)^{2}+(j-y)^{2}\right)^{\alpha m/2}}.
	\end{align*}
	Using the transformation (\ref{eq16}) again, one obtains
	\begin{align*}
		D_{T_{1},T_{2}}^{(2)}&=-2 m!\tilde{T}^{4-m\alpha}
		\int_{[0,T_{1}\tilde{T}^{-1}]}\int_{[0,T_{2}\tilde{T}^{-1}]}\sum_{i=0}^{T_{1}-1}\sum_{j=0}^{T_{2}-1}\dfrac{g(\tilde{T}u_{1},\tilde{T}u_{2})g\left(\dfrac{i}{\tilde{T}}\tilde{T},\dfrac{j}{\tilde{T}}\tilde{T}\right)}{\left(\left(\dfrac{i}{\tilde{T}}-u_{1}\right)^{2}+\left(\dfrac{j}{\tilde{T}}-u_{2}\right)^{2}\right)^{\alpha m/2}}\notag\\
		&\times\dfrac{L^{m}\left(\tilde{T}\bigg\Vert\left(\dfrac{i}{\tilde{T}}-u_{1},\dfrac{j}{\tilde{T}}-u_{2}\right)\bigg\Vert\right)du_{1}du_{2}}{\tilde{T}^{2}}.
	\end{align*}
	Multiplying and dividing by $g^{2}(\tilde{T},\tilde{T})$, we obtain
	\begin{align*}
		D_{T_{1},T_{2}}^{(2)}&=-2 m!\tilde{T}^{4-m\alpha}g^{2}(\tilde{T},\tilde{T})\\
		&\times\int_{[0,T_{1}\tilde{T}^{-1}]}\int_{[0,T_{2}\tilde{T}^{-1}]}\sum_{i=0}^{T_{1}-1}\sum_{j=0}^{T_{2}-1}\dfrac{g(\tilde{T}u_{1},\tilde{T}u_{2})g\left(\dfrac{i}{\tilde{T}}\tilde{T},\dfrac{j}{\tilde{T}}\tilde{T}\right)}{\left(\left(\dfrac{i}{\tilde{T}}-u_{1}\right)^{2}+\left(\dfrac{j}{\tilde{T}}-u_{2}\right)^{2}\right)^{\alpha m/2}}\\
		&\times\dfrac{L^{m}\left(\tilde{T}\bigg\Vert\left(\dfrac{i}{\tilde{T}}-u_{1},\dfrac{j}{\tilde{T}}-u_{2}\right)\bigg\Vert\right)du_{1}du_{2}}{g^{2}(\tilde{T},\tilde{T})\tilde{T}^{2}}.
	\end{align*}
	Again adding and subtracting either $g^{*}(u_{1},u_{2})$ or $g^{*}\left(\dfrac{i}{\tilde{T}},\dfrac{j}{\tilde{T}}\right)$ inside the integrals, we can write
	\begin{align}\label{eq22}
		D_{T_{1},T_{2}}^{(2)}=-2 m!\tilde{T}^{4-m\alpha}g^{2}(\tilde{T},\tilde{T})\left(\hat{I_{1}}+\hat{I_{2}}+\hat{I_{2}}'+\hat{I_{3}}\right),
	\end{align}
	where
	\begin{align*}
		\hat{I_{1}}&=\int_{[0,T_{1}\tilde{T}^{-1}]}\int_{[0,T_{2}\tilde{T}^{-1}]}\sum_{i=0}^{T_{1}-1}\sum_{j=0}^{T_{2}-1}\left[\dfrac{g(\tilde{T}u_{1},\tilde{T}u_{2})}{g\left(\tilde{T},\tilde{T}\right)}-g^{*}(u_{1},u_{2})\right]\\
		&\times\left[g\left(\dfrac{i}{\tilde{T}}\tilde{T},\dfrac{j}{\tilde{T}}\tilde{T}\right)-g^{*}\left(\dfrac{i}{\tilde{T}},\dfrac{j}{\tilde{T}}\right)\right]\frac{L^{m}\left(\tilde{T}\bigg\Vert\left(\dfrac{i}{\tilde{T}}-u_{1},\dfrac{j}{\tilde{T}}-u_{2}\right)\bigg\Vert\right)du_{1}du_{2}}{\tilde{T}^{2}\left(\left(\dfrac{i}{\tilde{T}}-u_{1}\right)^{2}+\left(\dfrac{j}{\tilde{T}}-u_{2}\right)^{2}\right)^{\alpha m/2}},\\
		\hat{I_{2}}&=\int_{[0,T_{1}\tilde{T}^{-1}]}\int_{[0,T_{2}\tilde{T}^{-1}]}\sum_{i=0}^{T_{1}-1}\sum_{j=0}^{T_{2}-1}\left[\dfrac{g(\tilde{T}u_{1},\tilde{T}u_{2})}{g\left(\tilde{T},\tilde{T}\right)}-g^{*}(u_{1},u_{2})\right]g^{*}\left(\dfrac{i}{\tilde{T}},\dfrac{j}{\tilde{T}}\right)\\
		&\times\frac{L^{m}\left(\tilde{T}\bigg\Vert\left(\dfrac{i}{\tilde{T}}-u_{1},\dfrac{j}{\tilde{T}}-u_{2}\right)\bigg\Vert\right)du_{1}du_{2}}{\tilde{T}^{2}\left(\left(\dfrac{i}{\tilde{T}}-u_{1}\right)^{2}+\left(\dfrac{j}{\tilde{T}}-u_{2}\right)^{2}\right)^{\alpha m/2}},\\
		\hat{I_{2}}^{'}&=\int_{[0,T_{1}\tilde{T}^{-1}]}\int_{[0,T_{2}\tilde{T}^{-1}]}\sum_{i=0}^{T_{1}-1}\sum_{j=0}^{T_{2}-1}g^{*}(u_{1},u_{2})\left[g\left(\dfrac{i}{\tilde{T}}\tilde{T},\dfrac{j}{\tilde{T}}\tilde{T}\right)-g^{*}\left(\dfrac{i}{\tilde{T}},\dfrac{j}{\tilde{T}}\right)\right]\\
		&\times\frac{L^{m}\left(\tilde{T}\bigg\Vert\left(\dfrac{i}{\tilde{T}}-u_{1},\dfrac{j}{\tilde{T}}-u_{2}\right)\bigg\Vert\right)du_{1}du_{2}}{\tilde{T}^{2}\left(\left(\dfrac{i}{\tilde{T}}-u_{1}\right)^{2}+\left(\dfrac{j}{\tilde{T}}-u_{2}\right)^{2}\right)^{\alpha m/2}},
	\end{align*}
	and
	\begin{align*}
		\hat{I_{3}}&=\int_{[0,T_{1}\tilde{T}^{-1}]}\int_{[0,T_{2}\tilde{T}^{-1}]}\sum_{i=0}^{T_{1}-1}\sum_{j=0}^{T_{2}-1}g^{*}(u_{1},u_{2})g^{*}\left(\dfrac{i}{\tilde{T}},\dfrac{j}{\tilde{T}}\right)\\
		&\times\frac{L^{m}\left(\tilde{T}\bigg\Vert\left(\dfrac{i}{\tilde{T}}-u_{1},\dfrac{j}{\tilde{T}}-u_{2}\right)\bigg\Vert\right)du_{1}du_{2}}{\tilde{T}^{2}\left(\left(\dfrac{i}{\tilde{T}}-u_{1}\right)^{2}+\left(\dfrac{j}{\tilde{T}}-u_{2}\right)^{2}\right)^{\alpha m/2}}.
	\end{align*}
	Similarly to the upper bounds for $I_{1}$ and $I_{2}$ we can estimate the terms $\hat{I_{1}},\hat{I_{2}},\hat{I_{2}}',$ and $\hat{I_{3}}$ as
	\begin{align*}
		\hat{I_{1}}&\leq\int_{[0,T_{1}\tilde{T}^{-1}]}\int_{[0,T_{2}\tilde{T}^{-1}]}\sum_{i=0}^{T_{1}-1}\sum_{j=0}^{T_{2}-1}\bigg\vert\dfrac{g(\tilde{T}u_{1},\tilde{T}u_{2})}{g\left(\tilde{T},\tilde{T}\right)}-g^{*}(u_{1},u_{2})\bigg\vert\\
		&\times\bigg\vert g\left(\dfrac{i}{\tilde{T}}\tilde{T},\dfrac{j}{\tilde{T}}\tilde{T}\right)-g^{*}\left(\dfrac{i}{\tilde{T}},\dfrac{j}{\tilde{T}}\right)\bigg\vert\frac{L^{m}\left(\tilde{T}\bigg\Vert\left(\dfrac{i}{\tilde{T}}-u_{1},\dfrac{j}{\tilde{T}}-u_{2}\right)\bigg\Vert\right)du_{1}du_{2}}{\tilde{T}^{2}\left(\left(\dfrac{i}{\tilde{T}}-u_{1}\right)^{2}+\left(\dfrac{j}{\tilde{T}}-u_{2}\right)^{2}\right)^{\alpha m/2}}\\
		&\leq\int_{\left[0,1\right]^{2}}\sum_{i=0}^{T_{1}-1}\sum_{j=0}^{T_{2}-1}\bigg\vert\dfrac{g(\tilde{T}u_{1},\tilde{T}u_{2})}{g(\tilde{T},\tilde{T})}-g^{*}(u_{1},u_{2})\bigg\vert\bigg\vert g\left(\dfrac{i}{\tilde{T}}\tilde{T},\dfrac{j}{\tilde{T}}\tilde{T}\right)-g^{*}\left(\dfrac{i}{\tilde{T}},\dfrac{j}{\tilde{T}}\right)\bigg\vert\\
		&\times\frac{L^{m}\left(\tilde{T}\bigg\Vert\left(\dfrac{i}{\tilde{T}}-u_{1},\dfrac{j}{\tilde{T}}-u_{2}\right)\bigg\Vert\right)du_{1}du_{2}}{\tilde{T}^{2}\left(\left(\dfrac{i}{\tilde{T}}-u_{1}\right)^{2}+\left(\dfrac{j}{\tilde{T}}-u_{2}\right)^{2}\right)^{\alpha m/2}}
		\leq\sup_{\substack{u_{1},u_{2}\in[0,1]}}\bigg\vert\dfrac{g(\tilde{T}u_{1},\tilde{T}u_{2})}{g(\tilde{T},\tilde{T})}-g^{*}(u_{1},u_{2})\bigg\vert^{2}\\
		&\times\sum_{i=0}^{T_{1}-1}\sum_{j=0}^{T_{2}-1}\int_{\left[0,1\right]^{2}}\frac{L^{m}\left(\tilde{T}\bigg\Vert\left(\dfrac{i}{\tilde{T}}-u_{1},\dfrac{j}{\tilde{T}}-u_{2}\right)\bigg\Vert\right)du_{1}du_{2}}{\tilde{T}^{2}\left(\left(\dfrac{i}{\tilde{T}}-u_{1}\right)^{2}+\left(\dfrac{j}{\tilde{T}}-u_{2}\right)^{2}\right)^{\alpha m/2}}
		\leq o(1)L^{m}(\tilde{T}),\quad \tilde{T}\longrightarrow\infty.
	\end{align*}
	Also we have
	\begin{align*}
		\hat{I_{2}}&\leq\int_{[0,T_{1}\tilde{T}^{-1}]}\int_{[0,T_{2}\tilde{T}^{-1}]}\sum_{i=0}^{T_{1}-1}\sum_{j=0}^{T_{2}-1}\bigg\vert\dfrac{g(\tilde{T}u_{1},\tilde{T}u_{2})}{g\left(\tilde{T},\tilde{T}\right)}-g^{*}(u_{1},u_{2})\bigg\vert\bigg\vert g^{*}\left(\dfrac{i}{\tilde{T}},\dfrac{j}{\tilde{T}}\right)\bigg\vert\\
		&\times\frac{L^{m}\left(\tilde{T}\bigg\Vert\left(\dfrac{i}{\tilde{T}}-u_{1},\dfrac{j}{\tilde{T}}-u_{2}\right)\bigg\Vert\right)du_{1}du_{2}}{\tilde{T}^{2}\left(\left(\dfrac{i}{\tilde{T}}-u_{1}\right)^{2}+\left(\dfrac{j}{\tilde{T}}-u_{2}\right)^{2}\right)^{\alpha m/2}}\\
		&\leq\int_{\left[0,1\right]^2}\sum_{i=0}^{T_{1}-1}\sum_{j=0}^{T_{2}-1}\bigg\vert\dfrac{g(\tilde{T}u_{1},\tilde{T}u_{2})}{g\left(\tilde{T},\tilde{T}\right)}-g^{*}(u_{1},u_{2})\bigg\vert\bigg\vert g^{*}\left(\dfrac{i}{\tilde{T}},\dfrac{j}{\tilde{T}}\right)\bigg\vert\\
		&\times\frac{L^{m}\left(\tilde{T}\bigg\Vert\left(\dfrac{i}{\tilde{T}}-u_{1},\dfrac{j}{\tilde{T}}-u_{2}\right)\bigg\Vert\right)du_{1}du_{2}}{\tilde{T}^{2}\left(\left(\dfrac{i}{\tilde{T}}-u_{1}\right)^{2}+\left(\dfrac{j}{\tilde{T}}-u_{2}\right)^{2}\right)^{\alpha m/2}}
		\leq\sup_{\substack{u_{1},u_{2}\in[0,1]}}\bigg\vert\dfrac{g(\tilde{T}u_{1},\tilde{T}u_{2})}{g(\tilde{T},\tilde{T})}-g^{*}(u_{1},u_{2})\bigg\vert\\
		&\times\int_{\left[0,1\right]^{2}}\sum_{i=0}^{T_{1}-1}\sum_{j=0}^{T_{2}-1}\frac{\bigg\vert g^{*}\left(\dfrac{i}{\tilde{T}},\dfrac{j}{\tilde{T}}\right)\bigg\vert L^{m}\left(\tilde{T}\bigg\Vert\left(\dfrac{i}{\tilde{T}}-u_{1},\dfrac{j}{\tilde{T}}-u_{2}\right)\bigg\Vert\right)du_{1}du_{2}}{\tilde{T}^{2}\left(\left(\dfrac{i}{\tilde{T}}-u_{1}\right)^{2}+\left(\dfrac{j}{\tilde{T}}-u_{2}\right)^{2}\right)^{\alpha m/2}}\\
		&\leq o(1)L^{m}(\tilde{T}),\quad \tilde{T}\longrightarrow\infty.
	\end{align*}
	Similarly $\hat{I_{2}}'\leq o(1)L^{m}(\tilde{T}),\tilde{T}\longrightarrow\infty$. Also, analogously to the case of $I_{3},$ we have $\hat{I_{3}}\sim l_{1,2}L^{m}(\tilde{T}),\quad \tilde{T}\longrightarrow\infty$.
	
	By combining these results and (\ref{eq22}), we have
	\begin{align*}
		D_{T_{1},T_{2}}^{(2)}=-2m!\tilde{T}^{4-m\alpha}L^{m}(\tilde{T})g^{2}(\tilde{T},\tilde{T})(l_{1,2}+o(1)),\qquad \tilde{T}\longrightarrow\infty.
	\end{align*}
	Using similar arguments as for the sums in $D^{(2)}_{T_{1},T_{2}}$ we obtain
	\begin{align*}
		D_{T_{1},T_{2}}^{(3)}=m!\tilde{T}^{4-m\alpha}L^{m}(\tilde{T})g^{2}(\tilde{T},\tilde{T})(l_{1,2}+o(1)),\qquad \tilde{T}\longrightarrow\infty.
	\end{align*}
	Finally, combining all the previous results, we get the statement of Theorem~\ref{theo3}.
\end{proof}

\section{Multidimensional case and applications}
This section gives a multidimensional version of Theorem~\ref{theo3}. It also demonstrates how Theorems \ref{theo3} and \ref{theo4} can be applied to obtain limit theorems for additive functionals that are analogous to the result in Theorem~\ref{theo2}.

Denote $\textit{\textbf{1}}_{n}=(1,\dots,1)\in\mathbb{R}^{n}$.	
\begin{assumption}\label{ass4}
	Let $g(\textbf{t}),\textbf{t}\in\mathbb{R}^{n}$, be such a function that $T^{2n-m\alpha}g^{2}(T\textbf{1}_{n})L^{m}(T)\longrightarrow\infty$, as $T\longrightarrow\infty,$ and there exists a function $g^{*}(\textbf{t}),\textbf{t}\in\mathbb{R}^{n}$, such that
	$$\lim_{T\rightarrow\infty}\bigg\vert\dfrac{g\left(T\textbf{t}\right)}{g\left(T\textbf{1}_{n}\right)}-g^{*}(\textbf{t})\bigg\vert\rightarrow 0$$
	uniformly on $\textbf{t}\in [0,1]^{n}$.
\end{assumption}
Let us consider $\xi(\textbf{t}),$ $\textbf{t}\in\mathbb{R}^{n}$. It is obvious that steps analogous to ones in Section~3 can be used to obtain a multidimensional version of Theorem \ref{theo3}.  
\begin{theorem}\label{theo4}
	Let $\tilde{T}=\max({T_{1},\dots,T_{n}})$. If Assumptions~{\rm\ref{ass1}} and {\rm\ref{ass4}} hold, and there exist 
	$\lim_{\substack{\tilde{T}\rightarrow\infty}}T_{l}/\tilde{T}, l=1,\dots,n,$ then 
	\begin{align*}
		\lim_{\substack{\tilde{T}\rightarrow\infty}}\dfrac{\mathbb{E}\left[\int_{\prod_{i=1}^{n}\left[0,T_{i}\right]}g(\textbf{t})H_{m}(\xi(\textbf{t}))d\textbf{t}-\sum_{i_{1}=0}^{[T_{1}]-1}\cdots\sum_{i_{n}=0}^{[T_{n}]-1}g(\textbf{i})H_{m}(\xi(\textbf{i}))\right]^{2}}{\tilde{T}^{2n-\alpha m}g^{2}(\tilde{T}\textbf{1}_{n})L^{m}(\tilde{T})}=0,
	\end{align*}
	where $0<\alpha<n/m, \textbf{i}=(i_{1},\dots,i_{n})$.
\end{theorem}

Let us consider the case when $\Delta$ is the multidimensional rectangle $\sqsubset\!\sqsupset_{\textbf{a,b}}=\{\textbf{x}\in\mathbb{R}^{n}:x_{l}\in\left[a_{l},b_{l}\right],l=1,\dots,n\},$ where $\textbf{a,b}\in\mathbb{R}^{n}$ such that $a_{l}<0<b_{l},l=1,\dots,n$.
\begin{theorem}\label{theo5}
	If Assumptions {\rm \ref{ass1}, \ref{ass2},} and {\rm \ref{ass4}} hold and $\alpha\in(0,n/m),$ then for $T\longrightarrow\infty$ the additive functional
	\begin{align*}
		\tilde{X}_{m}^{*}(T)=\dfrac{1}{T^{n-m\alpha/2}L^{m}(T)}\sum_{\textbf{i}\in\sqsubset\!\sqsupset_\textbf{a,b}(T)\cap\mathbb{Z}^{n}}H_{m}(\xi(\textbf{i}))
	\end{align*}
	converges weakly to the random variable $X_{m}^{*}(\sqsubset\!\sqsupset_{\textbf{a,b}})$ given by {\rm (\ref{eq11})} with 
	\begin{align*}
		K_{\sqsubset\!\sqsupset_{\textbf{a,b}}}(x)=\prod_{j=1}^{n}\dfrac{e^{ib_{j}x_{j}}-e^{ia_{j}x_{j}}}{ix_{j}}.
	\end{align*}
\end{theorem}
\begin{proof}
	By (\ref{eq011}) we obtain
	\begin{align*}
		K_{\sqsubset\!\sqsupset_\textbf{a,b}}(x)=\int_{a_{1}}^{b_{1}}\cdots\int_{a_{n}}^{b_{n}}e^{i\sum_{j=1}^{n} u_{j}x_{j}}du_{n}\cdots du_{1}=\prod_{j=1}^{n}\dfrac{e^{ib_{j}x_{j}}-e^{ia_{j}x_{j}}}{ix_{j}}.
	\end{align*}
	Therefore, by the proof of Theorem (\ref{theo2}) in \cite{leonenko2014sojourn}
	\begin{align}\label{eq24}
		X^{*}_{m}(T)=T^{\frac{m\alpha}{2}-n}L^{-\frac{m}{2}}(T)\int_{a_{1}T}^{b_{1}T}\cdots\int_{a_{n}T}^{b_{n}T}H_{m}(\xi(\textbf{t}))d\textbf{t}\xrightarrow{MSE}X^{*}_{m}(\sqsubset\!\sqsupset_\textbf{a,b}),
	\end{align}
	as $T\rightarrow\infty$.
	
	Note, that to obtain the result of the Theorem it is sufficient to prove that 
	\begin{align*}
		R(T)=\mathbb{E}\left[\tilde{X}_{m}^{*}(T)-X_{m}(\sqsubset\!\sqsupset_{\textbf{a,b}})\right]^{2}=0,\quad T\rightarrow\infty.
	\end{align*}
	One can estimate $R(T)$ as
	\begin{align}\label{eq25}
		R(T)&=\mathbb{E}\left[\tilde{X}_{m}^{*}(T)-X_{m}^{*}(T)+X_{m}^{*}(T)-X_{m}(\sqsubset\!\sqsupset_{\textbf{a,b}})\right]^{2}\notag\\
		&\leq 2\mathbb{E}\left[\tilde{X}_{m}^{*}(T)-X_{m}^{*}(T)\right]^{2}+2\mathbb{E}\left[X_{m}^{*}(T)-X_{m}(\sqsubset\!\sqsupset_{\textbf{a,b}})\right]^{2}.
	\end{align}
	The second term in (\ref{eq25}) approaches $0$ by (\ref{eq24}).
	Also, note that due to homogeneity of $\xi(\textbf{t})$ the  results of Theorems \ref{theo3} and \ref{theo4} are true if one simultaneously changes  $\int_{0}^{T_l}$ and $\sum_{i_{l}=0}^{\left[T_l\right]-1}$ by $\int_{-T_l}^{0}$ and $\sum_{1+\lceil - T_l\rceil}^{0}$, where $\lceil x \rceil$ is the ceiling function of $x$. The multidimensional rectangle $\sqsubset\!\sqsupset_{\textbf{a,b}}(T)$ is an union of $2^{n}$ disjoint multidimensional rectangles having a common vertex at the origin and all edges connected to it are of the form $[0,T_l]$ or $[-T_l,0].$ Therefore,
	\begin{align*}
		\mathbb{E}\left[\tilde{X}_{m}^{*}(T)-X_{m}^{*}(T)\right]^{2}\leq 2^{n}\sum_{j=1}^{2^{n}}\mathbb{E}\left[\tilde{X}_{m,j}^{*}(T)-X_{m,j}(T)\right]^{2},
	\end{align*}
	where $\tilde{X}_{m,j}^{*}$    and $X_{m,j}$ are, respectively, an integral and a sum that correspond to the $j^{\text{th}}$ multidimensional rectangle in the union above. By Theorem~\ref{theo4}, selecting $g(\textbf{t})\equiv const,$ each term
	$\mathbb{E}\left[\tilde{X}_{m,j}^{*}(T)-X_{m,j}(T)\right]^{2}\rightarrow 0,$ when $T\rightarrow\infty.$
	
	Hence, $R(T)\rightarrow 0,$ when $T\rightarrow\infty,$ as it was required. 
\end{proof}
\section{Conclusion}
The main result of this paper is a generalisation of Lemma 1~\cite{leonenko2006weak} to the multidimensional case and a general class of long-range dependent fields. The result is useful in direct translating limit theorems from weighted integral functionals to additive functionals and vice versa. Note that the obtained results can be applied to more general setting than $\Delta(r)$ in Theorems \ref{theo2} and \ref{theo5} as $T_{i},i=1,\dots,n,$ can increase non-homothetically. An example is presented by applying the result to integrals of random fields with a constant weight functions over multidimensional rectangles. 

Some interesting problems and possible extensions that we plan to address in future research are:
\begin{itemize}
	\item to derive similar result for the case of non-rectangular $\Delta$    and corresponding sums;
	\item to investigate the rate of convergence in Theorems \ref{theo3} and \ref{theo4};
	\item to study functionals with weight functions that depend on $T_{i},i=1,\dots,n$.
\end{itemize}
\section{Acknowledgements }
The authors are grateful for the referee's careful reading of the paper and comments that helped to correct some misprints.


\begin{thebibliography}{1}
	\bibitem{anh2017rate} V. Anh, N. Leonenko, A. Olenko, V. Vaskovych, \textit{On rate of convergence in non-central limit theorems}, arXiv:1703.05900, (2017).
	
	\bibitem{anh2015rate} V. Anh, N. Leonenko, A. Olenko, \textit{On the rate of convergence to Rosenblatt-type distribution}. Journal of Mathematical Analysis and Applications. \textbf{425} (2015), 111-132.
	
	\bibitem{bai2013multivariate} S. Bai, M. Taqqu, \textit{Multivariate limit theorems in the context of long‐range dependence}. Journal of Time Series Analysis. \textbf{34} (2013), 717-743.
	
	\bibitem{beran1994statistics} J. Beran, \textit{Statistics for long-memory processes}. CRC press, (1994).
	
	\bibitem{bingham1989regular} N. Bingham, C. Goldie, J. Teugels, \textit{Regular variation}. Cambridge university press, (1989).
	
	\bibitem{buldygin1992works} V. Buldygin, Y. Kozachenko, N. Leonenko, \textit{Works of MI Yadrenko in the theory of random fields}. Ukrainian Mathematical Journal. \textbf{441} (1992), 1343-1348.
	
	\bibitem{dobrushin1979non} R. Dobrushin, P. Major, \textit{Non-central limit theorems for non-linear functional of Gaussian fields}. Probability Theory and Related Fields. \textbf{50} (1979), 27-52.
	
	\bibitem{doukhan2002theory}
	P. Doukhan, G. Oppenheim, M. Taqqu, \textit{Theory and applications of long-range dependence}. Springer, (2002).
	
	\bibitem{doukhan2002long} P. Doukhan, M. Taqqu, G. Oppenheim, \textit{Long-range dependence: theory and applications}. Birkhauser, (2002).
	
	\bibitem{ivanov1989statistical} A. Ivanov, N. Leonenko, \textit{Statistical Analysis of Random Fields}. Dordrecht: Kluwer  Academic Publisher, (1989). 
	
	\bibitem{leonenko1999limit} N. Leonenko, \textit{Limit theorems for random fields with singular spectrum}. Springer, (1999).
	
	\bibitem{leonenko2014sojourn} N. Leonenko, A. Olenko, \textit{Sojourn measures of Student and Fisher-Snedecor random fields}. Bernoulli. \textbf{20} (2014), 1454-1483.
	
	\bibitem{leonenko2006weak} N. Leonenko, E. Taufer, \textit{Weak convergence of functionals of stationary long memory processes to Rosenblatt-type distributions}. Journal of Statistical Planning and Inference. \textbf{136} (2006), 1220-1236.
	
	\bibitem{major1981multiple} P. Major, \textit{Multiple Wiener-It{\^o} Integrals: With Applications to Limit Theorems}. Springer-Verlag, (2014).
	
	\bibitem{olenko2013limit} A. Olenko, \textit{Limit theorems for weighted functionals of cyclical long-range dependent random fields}. Stochastic Analysis and Applications. \textbf{31} (2013), 199-213.
	
	\bibitem{olenko2010limit} A. Olenko, B. Klykavka, \textit{A limit theorem for random fields with a singularity in the spectrum}. Theory of Probability and Mathematical Statistics. \textbf{81} (2010), 147-158.
	
	\bibitem{pakkanen2016functional} M. Pakkanen, A. R\'{e}veillac, \textit{Functional limit theorems for generalized variations of the fractional Brownian sheet}. Bernoulli \textbf{22} (2016), 1671-1708.
	
	\bibitem{palma2007long} W. Palma, \textit{Long-memory time series: theory and methods}. John Wiley and Sons, (2007).
	
	\bibitem{pipiras2010regularization} V. Pipiras, M. Taqqu,  \textit{Regularization and integral representations of Hermite processes}. Statistics and probability letters. \textbf{80} (2010), 2014-2023.
	
	\bibitem{rosenblatt1981limit} M. Rosenblatt, \textit{Limit theorems for Fourier transforms of functionals of Gaussian sequences}. Probability theory and related fields. \textbf{55} (1981), 123-132.
	
	\bibitem{rosenblatt1961independence} M. Rosenblatt, \textit{Independence and dependence}. In Proc. 4th Berkeley sympos. math. statist. and prob. \textbf{2} (1961), 431-443.
	
	\bibitem{taqqu1979convergence} M. Taqqu, \textit{Convergence of integrated processes of arbitrary Hermite rank}. Probability Theory and Related Fields. \textbf{50} (1979), 53-83.
	
	\bibitem{taqqu1975weak} M. Taqqu, \textit{Weak convergence to fractional Brownian motion and to the Rosenblatt process}. Probability Theory and Related Fields. \textbf{31} (1975), 287-302.
	
	\bibitem{yadrenko1983spectral}
	M. Yadrenko, \textit{Spectral theory of random fields}. New York: Optimization Software, (1983).
\end{thebibliography}
\end{document}